\documentclass{article}
\usepackage[english]{babel}
\usepackage{amsmath,amsthm,amssymb ,latexsym, enumerate, hyperref, color, makeidx, xcolor, graphicx, amsmath, amsfonts, amsthm,amscd, amsmath, amssymb, amstext}
\allowdisplaybreaks
\newtheorem{theorem}{Theorem}
\newtheorem{corollary}[theorem]{Corollary}
\newtheorem{definition}[theorem]{Definition}
\newtheorem{example}[theorem]{Example}
\newtheorem{lemma}[theorem]{Lemma}

\newtheorem{proposition}[theorem]{Proposition}
\newtheorem{remark}[theorem]{Remark}

\def\supp{\mathop{\rm supp}\nolimits}

\begin{document}

\title{Idempotent approach to level-2 variational principles in Thermodynamical Formalism}

\author{ A. O. Lopes, J. K. Mengue and E. R. Oliveira}

\maketitle

\begin{abstract}
In this work we introduce an idempotent pressure to level-2 functions and its associated density entropy.   {All this is related to idempotent pressure functions which is the natural concept that corresponds to the meaning of probability in the level-2 max-plus context.} In this general framework the equilibrium states, maximizing the variational principle, are not unique. We investigate the connections with the general convex pressure introduced recently to level-1 functions by  Bi\'{s}, Carvalho, Mendes and Varandas. Our general setting contemplates the dynamical and not dynamical framework. We also study a characterization of the density entropy in order to get an idempotent pressure invariant by dynamical systems acting on probabilities;  this is therefore a level-2 result. We are able to produce idempotent pressure functions at level-2 which are invariant by the dynamics of the  pushforward map via a form of Ruelle operator.
\end{abstract}

\section{Introduction}\label{Idempotent  pressure functions}

In this paper we aim to define an idempotent analysis approach to level-2 variational principles in Thermodynamical Formalism.

The use of the Idempotent Analysis is a reasonable choice, because  as will see it encloses many of the classical constructions of thermodynamical formalism. Also, entropy and pressure are variational  concepts  which can be modeled in an idempotent framework, in the way that its properties will derive from the fundamental theorems of the Idempotent Analysis. The study of optimization problems was the motivation to the introduction of the Idempotent Analysis by Maslov in \cite{KM89}.

In \cite{MO1},  via the idempotent framework, the authors  addressed the issue of  idempotent measures  for place dependent idempotent iterated function systems. As a tool, a representation for idempotent probabilities on compact metric spaces is proved there and we will use it here. It is worth to notice that for bounded maps on separable locally compact topological spaces taking value in some semiring, tending to infinite at infinite and with compact support, the original Maslov's work \cite{KM89} considered linear functionals acting on that functions with image in a semiring. The setting and proofs  in \cite{MO1} are  somehow different from \cite{KM89}.

Among other things we will show that adapting the idempotent formalism of \cite{MO1} for the level-2 setting, we can use  idempotent probabilities to model variational principles associated to the nonlinear thermodynamic formalism framework. A level-1 property is related to points in a compact metric space $X$ and the variational principles are defined for continuous functions $\varphi: X \to \mathbb{R}$.  On the other hand, a level-2 property is related to points in $\mathcal{P}(X)$ (the space of probabilities on $X$ equipped with the weak$^*$ topology) and the variational principles are defined for continuous functions $g: \mathcal{P}(X) \to \mathbb{R}$. A simple example of continuous function $g: \mathcal{P}(X) \to \mathbb{R}$ is
$$g_{A}(\mu)=\int_{X} A d\mu, \, \\, \text{where}\, A:X \to \mathbb{R}\,\text{is continuous}.$$

In ergodic theory, questions at level-2 refer to properties related to the global study of the set of different probabilities on a given compact metric space $X$. For example, in \cite{L3} the author study large deviations in the set of probabilities over the symbolic space with a finite number of symbols, and minus entropy plays the role of a deviation function, while in \cite{LO1}, the authors study thermodynamic formalism, when the dynamics is given by the push-forward map acting in  $\mathcal{P}(X)$;  a form of   Ruelle operator is introduced and a kind of entropy was defined. In \eqref{prew} we present an operator that can  in some sense corresponds to the dual of the Ruelle operator on the ifs  level-2 setting. Other results related to the dynamics of the push-forward map appear in \cite{Fagner}, \cite{Bauer} and \cite{BVerm}.

 The thermodynamic formalism which we develop is closely related to the study of the Curie-Weiss model, which is of great importance in Statistical Mechanics, see \cite{FV}, \cite{Barre}, \cite{BKL}, \cite{LeWa1} and \cite{LeWa2}.

Basic properties on the Max-Plus algebra can be founded in chapter 6 in \cite{BLL} and \cite{Kol01} (see also Section 5.2 in \cite{G1} ).
The idempotent mathematics utilizes the idempotent semiring $\mathbb{R}_{\max}=\mathbb{R}\cup \{-\infty\}$ endowed with the operations $\oplus:=\max$ and $\odot:=+$. Note that $r \oplus r = \max(r, r)= r, \forall r \in \mathbb{R}_{\max}$ and  the neutral elements for $\oplus$ and $\odot$ are, respectively, $-\infty$ and $0$. Also, $\mathbb{R}_{\max}$  is not a ring since, for example, $2 \oplus r = -\infty$ has no solution with respect to $r$, so there is no symmetrical element ``-2".

Given a compact metric space $(X,d)$,  the space $\mathcal{P}(X)$ of probabilities on the Borel sigma-algebra of $X$ is a compact space with respect to the weak$^*$ topology. One can metrize such topology as in \cite{Villa}, by choosing the Wasserstein-1 (or Monge-Kantorovich) metric
\begin{equation} \label{WW1} W_{1}(\mu,\nu)= \max_{\operatorname{Lip}(f) \leq 1} \mu(f)-\nu(f).
\end{equation}
Indeed, it is widely known that $(\mathcal{P}(X), W_{1})$ is a compact metric space, as a consequence of the Banach-Alaoglu theorem.

We denote by $C(X,\mathbb{R})$ the space of continuous functions from $X$ to $\mathbb{R}$ and by $C(\mathcal{P}(X),\mathbb{R})$ the space of continuous functions on $\mathcal{P}(X)$ taking values at $\mathbb{R}$ (instead  $\mathbb{R}_{\max}$).

\begin{definition} \label{def :max-plus operations}
	Let $\mathbb{Y}=C(\mathcal{P}(X), \mathbb{R})$ be the space of continuous functions on $\mathcal{P}(X)$.  Let us define the max-plus linear operations
\begin{itemize}
  \item $\bigoplus: \mathbb{Y} \times \mathbb{Y} \to \mathbb{Y}$ given by $(g \bigoplus g')(\mu) = g(\mu) \oplus g'(\mu)$;
  \item $ \bigodot:  \mathbb{R} \times \mathbb{Y} \to \mathbb{Y}$ given by $(c \bigodot g)(\mu) = c \odot g(\mu)$,
\end{itemize}
for $c \in \mathbb{R}$ and $g,g' \in \mathbb{Y}$.
\end{definition}

\begin{definition}\label{def: variation metric}
An (level-2) \textbf{idempotent  pressure function}, is a max-plus linear functional $\ell: C(\mathcal{P}(X), \mathbb{R}) \to \mathbb{R}$, that is, the following axioms are fulfilled for any $c \in \mathbb{R}$ and $g,g' \in C(\mathcal{P}(X), \mathbb{R})$
\begin{itemize}
  \item Axiom A1 \begin{equation}\label{eq:A2}
              \ell(c \bigodot g)=c \odot \ell(g)
            \end{equation}
  \item Axiom A2 \begin{equation}\label{eq:A3}
              \ell(g\bigoplus g')=\ell(g) \oplus \ell(g').
            \end{equation}
\end{itemize}
\end{definition}

In the following theorem we present a paramount result characterizing  idempotent pressure functions, which is a max-plus analogous of the Riesz representation theorem.

\begin{theorem} \label{thm: represetation variation} Let $(X,d)$ be a compact metric space and $\mathcal{P}(X)$ be the set of probabilities over the Borel sigma algebra. Consider $\mathcal{P}(X)$ as a metric space with any metric equivalent to the weak-* topology.

If  $\ell: C(\mathcal{P}(X), \mathbb{R})\to\mathbb{R}$ is an  idempotent pressure function, then there exists a unique upper semi-continuous (u.s.c.) function $h: \mathcal{P}(X) \to \mathbb{R}_{\max}$ such that
\begin{equation}\label{eq:represent variation}
  \ell(g)= \sup_{\mu \in \mathcal{P}(X)} [ g(\mu)+h(\mu)],
\end{equation}
for any $g\in C(\mathcal{P}(X), \mathbb{R})$. Reciprocally, if $h:\mathcal{P}(X) \to \mathbb{R}_{\max}$ is bounded above and it is not identically $-\infty$ then equation \eqref{eq:represent variation} defines an idempotent pressure function.
\end{theorem}

With the introduction of above result the following definition is natural.

\begin{definition}
	Let $(X,d)$ be a compact metric space, $\ell: C(\mathcal{P}(X), \mathbb{R}) \to \mathbb{R}$ be an idempotent pressure and $h_{\ell}: \mathcal{P}(X) \to \mathbb{R}$ be the unique u.s.c. function satisfying \eqref{eq:represent variation}.
	We say that $h_{\ell}$ is the \textbf{density entropy} associated to the idempotent pressure function $\ell$.
	Moreover,  given $\ell$, we call any probability $\nu \in \mathcal{P}(X)$ attaining the supremum, that is,
	\begin{equation}\label{eq:generalized max-plus equilibrium}
		P_{\ell}(g)= h_{\ell}(\nu) + g(\nu),
	\end{equation}
	an \textbf{equilibrium state} associated to the idempotent pressure function $\ell$.
\end{definition}

We remark that the existence of equilibrium states associated to the idempotent pressure function $\ell$ is a consequence of $\mathcal{P}(X)$ to be compact and the map $\mu \mapsto [g(\mu)+h_{\ell}(\mu)]$ to be u.s.c.

In \cite{Bis} the authors consider a generalization of thermodynamic formalism via concave analysis; not exactly aimed to analyze level-2 questions.
We will explain the connection between the setting in \cite{Bis} and  the present level-2 setting  in Section  \ref{sec : level1}.
In section \ref{sec: proof1}, which consider idempotent measures in quite general terms  we prove Theorem \ref{thm: represetation variation}. {In Section~\ref{sec: examples} we present some examples illustrating the connections with classical constructions in ergodic theory and thermodynamic formalism.} It is simple to exhibit in our setting examples where the equilibrium state is not unique and moreover the set of equilibrium states is not convex. We exhibit also how the idempotent pressure and its variational characterization can be used to model non-linear dependencies of the potential. In section \ref{sec: invariant} we present characterizations of the entropy in order to get an idempotent pressure invariant for a dynamical system or a transfer operator. {In Section~\ref{sec:The inverse problem} we analyse the inverse problem of finding an max-plus IFS for which a given idempotent pressure function is invariant. We also study in Section~\ref{sec:Max-plus dynamics} some max-plus dynamical aspects involving the max-plus averages of a dynamical system.}

\section{Proof of Theorem \ref{thm: represetation variation}}\label{sec: proof1}

A direct proof would require a large amount of work and sophisticated arguments. However, we can rely on the Idempotent Analysis results.  Given  a compact metric space $(Z,d)$ consider the set $C(Z,\mathbb{R})$ of continuous functions on $Z$.

\begin{definition} \label{def: Maslov measure}
  A function $m: C(Z,\mathbb{R}) \to \mathbb{R}$ is an idempotent (or Maslov) measure over $Z$ if
  \begin{itemize}
  \item $m(c\odot f) = c \odot m(f)$, $c\in \mathbb{R}$ and $f\in C(Z,\mathbb{R})$;
  \item $m(f \oplus f') =m(f) \oplus m(f')$, $f, f' \in C(Z,\mathbb{R})$.
\end{itemize}
The set of all idempotent measures is the max-plus dual of $C(Z,\mathbb{R})$. The set of idempotent probabilities over $Z$, denoted $I(Z)$,  is the set of all idempotent measures $m$ satisfying
$m(0)=0.$
\end{definition}

The next result was proved in \cite{BRZ}.
\begin{theorem}\cite{BRZ} $I(Z)$ endowed with the pointwise convergence topology  is compact.
\end{theorem}

\begin{definition}\label{def:usc}
   We denote $U(Z)$ the set of all u.s.c. functions $\lambda$ taking image in $\mathbb{R}_{\max}$ such that $\lambda(z_0)>-\infty$ for some $z_0 \in Z$. In other words $\operatorname{supp}(\lambda)=\{ z| \lambda(z)>-\infty\} \neq \varnothing$.
\end{definition}

\begin{remark}
	In Definition~\ref{def :max-plus operations} the algebraic structure $(\mathbb{Y},  \bigoplus, \bigodot)$  is not a $\mathbb{R}_{\max}$-semimodule (neither a $\mathbb{R}$-semimodule). To start, it is not closed under the  scalar multiplication because $-\infty \bigodot g \not\in C(\mathcal{P}(X), \mathbb{R})$. Moreover, even if take $1_{\mathbb{Y}}(y)=0, \forall y \in Y$ in such way that $(\mathbb{Y}=C(\mathcal{P}(X), \mathbb{R}),\bigodot)$  is a monoid the set $(\mathbb{Y},\bigoplus)$ is just a semigroup. Indeed,  the element candidate to be the identity $0_{\mathbb{Y}}$, must verify $(g\bigoplus 0_{\mathbb{Y}})(y) =g(y) \oplus 0_{\mathbb{Y}}(y)= g(y), \forall y \in Y$ is the function
	$$0_{\mathbb{Y}}(y):=-\infty,\; \forall y \in Y$$
	meaning that $0_{\mathbb{Y}} \in C(Y, \mathbb{R}_{\max})$ but $0_{\mathbb{Y}} \not\in C(Y, \mathbb{R})=\mathbb{Y}$.

\end{remark}

Despite these differences, a representation theorem for idempotent measures similar to \cite{KM89} was proved in \cite{MO1} for this particular setting.

\begin{theorem}\cite[Theorem 1.2]{MO1} \label{thm:represent idepotent}
  Let $(Z,d)$ be a compact metric space. A function $m:C(Z,\mathbb{R})\to\mathbb{R}$ is an idempotent measure if, and only if,
  \begin{equation}\label{eq:represent idepotent}
  m(f)=\sup_{z \in Z} [\lambda(z) + f(z)]
  \end{equation}
  where $ \lambda \in U(Z)$.  Moreover, such function is unique in $U(Z)$ and $m \in I(Z)$ if, and only if, $\sup_{z \in Z} \lambda(z)=0$.
\end{theorem}

 It is worth to notice that an analogous result was previously stated for separable locally compact topological spaces;  the original Maslov's work \cite{KM89}, considered functionals acting on continuous functions, tending to zero at infinite and with compact support, taking image in to a metric semiring. The setting of \cite{KM89} is not exactly the same as  in \cite{MO1}.

We can now prove Theorem~\ref{thm: represetation variation}
\begin{proof} [Proof of Theorem~\ref{thm: represetation variation}]
In our case, the hypothesis ensures that $Z= \mathcal{P}(X)$ is a compact metric space thus the  idempotent pressure function $\ell$ is actually an idempotent measure (or Maslov measure) on $\mathcal{P}(X)$. By Theorem~\ref{thm:represent idepotent} we obtain the representation in Equation~\eqref{eq:represent variation}. Reciprocally, if $h:\mathcal{P}(X) \to \mathbb{R}_{\max}$ is bounded above and it is not identically $-\infty$ then defining $\ell$ by  $ \ell(g)= \sup_{\mu \in \mathcal{P}(X)} [ g(\mu)+h(\mu)],$ it is immediate to check that we get an idempotent pressure functional.	
\end{proof}

\begin{remark}
	It can be constructed a version of Theorem~\ref{thm: represetation variation} for  idempotent pressure functions on topological spaces, see \cite[Theorem 1]{KM89}. In this case, the original Maslov's work \cite{KM89}, considered bounded maps on separable locally compact topological spaces taking value in some semiring, tending to infinite at infinite and with compact support. In that work they considered the semimodule of linear (max-plus) functionals acting on that functions with image in a metric semiring.
\end{remark}

\section{Comparison with level-1 convex pressure}\label{sec : level1}

For sake of comparison we recall that \cite{Bis} (and its correction\cite{BisCorrection}) defines a (convex, level-1) pressure function. We assume in the present work that $X$ is just  a compact metric space.
\begin{definition}\cite[Definition 2.1]{Bis}\label{def:Bis}
	A function $\Gamma: C(X, \mathbb{R}) \to \mathbb{R}$ is called a pressure function if it satisfies the following conditions:\\
	(C1) Monotonicity: $\varphi \leqslant \psi \Rightarrow \Gamma(\varphi) \leqslant \Gamma(\psi) \quad \forall \varphi, \psi \in C(X, \mathbb{R})$.\\
	(C2) Translation invariance: $\Gamma(\varphi+c)=\Gamma(\varphi)+c \quad \forall \varphi \in C(X, \mathbb{R}) \quad \forall c \in \mathbb{R}$.\\
	(C3) Convexity: $\Gamma(t \varphi+(1-t) \psi) \leqslant t \Gamma(\varphi)+(1-t) \Gamma(\psi) \quad \forall \varphi, \psi \in C(X, \mathbb{R}) \quad \forall t \in[0,1]$.
\end{definition}
Then, it is provide the following characterization (here applied for compact spaces):
\begin{theorem}\cite[Theorem 1]{BisCorrection}\label{thm:represent pressure varandas}
	Let $\Gamma:  C(X, \mathbb{R}) \rightarrow \mathbb{R}$ be a pressure function. Then
	\begin{equation}\label{eq: varandas entropy}
		\Gamma(\varphi)=\max _{\mu \in \mathcal{P}(X)}\left\{\int \varphi d \mu+\mathfrak{h}(\mu)\right\} \quad \forall \varphi \in C(X, \mathbb{R})
	\end{equation}
	where
	$$
	\mathfrak{h}(\mu)=\inf _{\varphi \in \mathcal{A}_{\Gamma}}\left\{\int \varphi d \mu\right\} \quad \text { and } \quad \mathcal{A}_{\Gamma}=\{\varphi \in C(X, \mathbb{R}): \Gamma(-\varphi) \leqslant 0\}.
	$$
	Moreover, $\mathfrak{h}(\mu)$ is concave and upper semi-continuous. Furthermore, if $\alpha: \mathcal{P}(X) \rightarrow \mathbb{R} \cup\{-\infty,+\infty\}$ is another function taking the role of $\mathfrak{h}$ in \eqref{eq: varandas entropy}, then $\alpha \leqslant \mathfrak{h}$. In addition, one has
	$
	\mathfrak{h}(\mu)=\inf _{\varphi \in C(X, \mathbb{R})}\left\{\Gamma(\varphi)-\int \varphi d \mu\right\}, \quad \forall \mu \in \mathcal{P}(X).
	$
	Finally, the maximum in \eqref{eq: varandas entropy} is attained in $\mathcal{P}(X)$.
\end{theorem}

\begin{remark} \label{rem:inclusion}
	We notice that there exists a canonical inclusion  $j: C(X, \mathbb{R}) \to C(\mathcal{P}(X), \mathbb{R})$ given by
	$$j(\varphi) (\mu) = \int_{X} \varphi(x) d\mu(x), \; \mu \in \mathcal{P}(X).$$
	In this particular case we have
	\begin{equation}\label{eq:potential max-plus equilibrium}
	\ell(j(\varphi)) = \sup_{\mu \in \mathcal{P}(X)} h_{\ell}(\mu) + \int_{X} \varphi(x) d\mu(x).
	\end{equation}
However it is important to observe that in $C(X, \mathbb{R})$ there exists a natural max-operation, that is:
$\max(\varphi,\psi)(x) := \max\{\varphi(x),\psi(x)\} = \varphi(x)\oplus \psi(x).$
	In general, this max-operation structure on $C(X, \mathbb{R})$ does not agree with the $\bigoplus$ of $C(\mathcal{P}(X), \mathbb{R})$, that is, $j(\max(\varphi,\psi))\neq j(\varphi)\oplus j(\psi)$. Precisely,
	$$j(\max(\varphi,\psi)) (\mu) = \int \max(\varphi,\psi) d\mu  > \max ( \int \varphi d\mu, \int \psi d\mu)= [j(\varphi)\oplus j(\psi)] (\mu).$$
	Consequently,
	$$\ell(j(\max(\varphi,\psi))) \neq \max\{\ell(j(\varphi)), \ell(j(\psi))\}.$$
	
We highlight the fact   that, when considering the inclusion map $j$, the correct $\oplus$ operation to be used is the level-2 max-operation, $$[j(\varphi)\oplus j(\psi)] (\mu) = \max \left( \int \varphi d\mu, \int \psi d\mu\right).$$
  	
\end{remark}

The next theorem shows that the restriction of a level-2 idempotent pressure function to functions of $C(X, \mathbb{R})$ is actually a level-1 convex pressure function in the sense of Definition \ref{def:Bis}. Precisely, the canonical inclusion $j: C(X, \mathbb{R}) \to C(\mathcal{P}(X),\mathbb{R})$ defines a projection of level-2 idempotent pressure functions to level-1 pressure functions.

\begin{theorem}
	Consider  an idempotent pressure function $\ell$ with density entropy $h_{\ell}$ and the canonical inclusion $j: C(X, \mathbb{R}) \to C(\mathcal{P}(X), \mathbb{R})$ given by $j(\varphi) (\mu) = \int_{X} \varphi(x) d\mu(x), \; \mu \in \mathcal{P}(X)$. If we define $\Gamma_{\ell}: C(X, \mathbb{R}) \to \mathbb{R}$ by
	$$\Gamma_{\ell}(\varphi):=\ell(j(\varphi))$$
	then $\Gamma_{\ell}$ is a convex pressure function (in the sense of Definition \ref{def:Bis}). Let $\mathfrak{h}$ be the  concave and upper semi-continuous (entropy) function given in Theorem \ref{thm:represent pressure varandas}. Then  $h_{\ell} \leq \mathfrak{h}$. Reciprocally, each pressure function (in the sense of Definition \ref{def:Bis}) is of the form $\Gamma_{\ell}$ where $\ell$ is an idempotent pressure function. Such map $\ell\mapsto \Gamma_{\ell}$ is surjective, but not injective.
\end{theorem}
\begin{proof}
	Let $\ell$ be an idempotent pressure function. From Theorem \ref{thm: represetation variation} there exists a unique u.s.c. function $h_\ell: \mathcal{P}(X) \to \mathbb{R}_{\max}$ such that
	\[	\ell(g)= \sup_{\mu \in \mathcal{P}(X)} [ g(\mu)+h_\ell(\mu)],\]
	for any $g\in  C(\mathcal{P}(X), \mathbb{R})$. Define $\Gamma_{\ell}(\varphi):=\ell(j(\varphi)), \; \varphi \in C(X, \mathbb{R})$. In this way,
	\[\Gamma_{\ell}(\varphi) =\sup_{\mu \in \mathcal{P}(X)} [ \int \varphi d\mu+h_\ell(\mu)].\]
	Consequently, the hypotheses in Definition \ref{def:Bis} are immediately satisfied by $\Gamma_{\ell}$. Furthermore, by Theorem~\ref{thm:represent pressure varandas}, $h_{\ell} \leq \mathfrak{h}$.
	
	Reciprocally, consider $\Gamma(\varphi)=\max_{\mu \in \mathcal{P}(X)}\left\{\mathfrak{h}(\mu)+\int \varphi d \mu\right\}$ a convex pressure function. Note that $\mathfrak{h}(\mu)$ is u.s.c., so the formula
	$\ell(g)= \sup_{\mu \in \mathcal{P}(X)} \mathfrak{h}(\mu) + g(\mu),\,g\in C(\mathcal{P}(X), \mathbb{R})$  defines an idempotent pressure function which extends $\Gamma$. Clearly $\Gamma_{\ell}=\Gamma$.  The map $\ell \mapsto \Gamma_\ell$ is not injective, because the unique u.s.c. density entropy in Theorem \ref{thm: represetation variation}  does not need to be concave.
\end{proof}

\begin{proposition} \label{kul}  Suppose that $h_\ell$ is u.s.c. and concave and it is the density entropy of an idempotent pressure $\ell$. Then for all $\mu \in \mathcal{P}(X)$ we have
	\begin{equation} \label{fu1} h_\ell(\mu) =  \inf_{g \in C(\mathcal{P}(X),\mathbb{R})} \{ \ell(g) - g(\mu)\}
	\end{equation}
and for all $g\in C(\mathcal{P}(X),\mathbb{R})$ we have
\[\ell(g) = \sup_{\mu \in \mathcal{P}(X)} \{h_\ell(\mu) +g(\mu)\}.\]
Finally, considering $\Gamma_\ell$ and its entropy $\mathfrak{h}$ (given in Theorem \ref{thm:represent pressure varandas}) we have $\mathfrak{h}=h_\ell$.
\end{proposition} 	

As a consequence of this proposition we get the following corollary.

\begin{corollary}
	Suppose that $\alpha: \mathcal{P}(X) \rightarrow \mathbb{R}_{\max} $ is a function taking the role of $\mathfrak{h}$ in \eqref{eq: varandas entropy}. If $\alpha$ is u.s.c and concave then $\alpha = \mathfrak{h}$.
\end{corollary}
\begin{proof}
Let $\Gamma$ be a convex pressure function and suppose that $\alpha: \mathcal{P}(X) \rightarrow \mathbb{R}_{\max} $ is u.s.c, concave and satisfies
\[\Gamma(\varphi)=\max _{\mu \in \mathcal{P}(X)}\left\{\int \varphi d \mu+\alpha(\mu)\right\} \quad \forall \varphi \in C(X, \mathbb{R}).\]
Let $\ell$ be the idempotent pressure with density entropy $\alpha$. We have $\Gamma=\Gamma_\ell$ and applying above proposition we get  $\mathfrak{h}=\alpha.$
\end{proof}

\begin{proof}[Proof of Proposition \ref{kul}] We will adapt the proof of Theorem 9.12 in \cite{Walters} to the present case.
	Equation
	\[\ell(g) = \sup_{\mu \in \mathcal{P}(X)} \{h_\ell(\mu) +g(\mu)\}\]
is a consequence of 	$h_\ell$ to be the density entropy of $\ell$ and its concavity is not necessary.

Given $\mu\in\mathcal{P}(X)$ and $g  \in C(\mathcal{P}(X),\mathbb{R})$, from above equation we get
	$$\ell(g) \geq h_\ell(\mu) + g(\mu).$$
Then
	$$\ell(g) - g(\mu)\geq h_\ell(\mu)$$
and consequently
	\begin{equation} \label{nnf1}\inf_{g \in C(\mathcal{P}(X),\mathbb{R})} \{\ell(g) - g(\mu)\}\geq h_\ell(\mu).
	\end{equation}

On the other hand, considering $\Gamma_{\ell}: C(X, \mathbb{R}) \to \mathbb{R}$, defined by
$\Gamma_{\ell}(\varphi):=\ell(j(\varphi))$ we get a convex pressure function. From, Theorem \ref{thm:represent pressure varandas}
\[\mathfrak{h}(\mu)=\inf_{\varphi \in C(X, \mathbb{R})}\left\{\Gamma_\ell(\varphi)-\int \varphi d \mu\right\}. \]
Then
\[\mathfrak{h}(\mu)= \inf_{\varphi \in C(X, \mathbb{R})}\left\{\ell(j(\varphi))-j(\varphi)(\mu)\right\}\geq\inf_{g \in C(\mathcal{P}(X),\mathbb{R})} \{\ell(g) - g(\mu)\}\geq h_\ell(\mu). \]
Now we will show that
\begin{equation}\label{fu2}
	h_\ell \geq \inf_{\varphi \in C(X, \mathbb{R})}\left\{\Gamma_\ell(\varphi)-\int \varphi d \mu\right\}.
\end{equation}

Denote
$$C = \{(\mu,t)\in \mathcal{P} (X) \times \mathbb{R} \,| \, t \leq h_\ell(\mu) \}. $$

Note that for any  $\mu$ in the support of  $h_\ell$, there exists a $t$ such that $(\mu,t)\in C.$
As $h_\ell$ is u.s.c. and concave by hypothesis we get that   $C$ is a convex closed set.  Fix a probability $\mu_0$ and take $b> h_\ell(\mu_0) $.
As $h_\ell$ is u.s.c., we get that $(\mu_0,b) \notin C$.  The set $\{(\mu_0,b) \}$ is compact and convex. It follows that there exists a continuous function $A: X \to \mathbb{R}$ and a real number  $\alpha$ (see Separation Theorem in  page 417 in \cite{DF}  or  Section 2 in \cite{Mord})  such that
$$ \int A  d \mu  +\alpha t > \int A  d \mu_0  +\alpha b,$$
for all $(\mu,t) \in C \subset  \mathcal{P}(X)\times \mathbb{R}.$

It follows that for any  probability  $\mu\in\mathcal{P}(X)$, taking $t=h_\ell(\mu)$,
$$ \int A  d \mu  +\alpha   h_\ell(\mu)> \int A  d \mu_0  +\alpha b.$$
Now, taking $\mu=\mu_0$ in the  above expression we get that $\alpha   h_\ell(\mu_0)> \alpha b.$ It follows that $\alpha<0.$ Therefore, for any $\mu\in \mathcal{P}(X)$
\begin{equation} \label{jet}  h_\ell(\mu)+ \frac{1}{\alpha}\int A d \mu < b + \frac{1}{\alpha}\int A d \mu_0.
\end{equation}
Taking the supremum on the left hand side with respect to $\mu\in\mathcal{P}(X)$ we get
$$\Gamma_\ell ( \frac{A}{\alpha} )\leq b +  \frac{1}{\alpha}\int A d \mu_0.$$

We conclude that, for any $b>   h_\ell(\mu_0),$
$$b \geq \Gamma_\ell ( \frac{A}{\alpha} ) -  \frac{1}{\alpha}\int A d \mu_0\geq \inf \{ \Gamma_\ell (\varphi) - \int \varphi d \mu_0\, |\, \varphi \in C(X, \mathbb{R}) \}.$$
Finally, we get that
$$  h_\ell(\mu_0) \geq \inf \{ \Gamma_\ell (\varphi) - \int \varphi d \mu_0\, |\, \varphi \in C(X, \mathbb{R}) \}.$$
This shows that  \eqref{fu2} is true and finishes the proof.

\end{proof}

\section{Some examples}\label{sec: examples}

Given a fixed density entropy $h$ and a fixed $g \in C(\mathcal{P}(X), \mathbb{R})$, one interesting problem is to find an $\mu \in \mathcal{P}(X)$ attaining the value   $\ell_h (g)$ (an equilibrium state).

\begin{example} \label{tf1} Take $X=\{1,...,d\}$ and then $\mathcal{P}(X)$ as the simplex
\begin{equation} \label{ewq1}\mathcal{P}(X) = \{ p=(p_1,p_2,...,p_d) \,|\, \sum_{j=1}^{d} p_j=1; \, p_j\geq 0 \,\forall j\in\{1,2,..,d\}\}.
\end{equation}
Consider the Shannon entropy
\begin{equation} \label{ewq}   h  (p_1,p_2,...,p_d)  = - \sum_{j=1}^{d} p_j \log p_j,
\end{equation}
where $0\cdot\log(0)=0$ by convention. This function $h$ is continuous and concave. Consider the functional $\ell_h$ as defined by \eqref{eq:represent variation}.  Then,
$\ell_h $ is an idempotent pressure function.

Let us consider the level-1 case. Take a vector $(g_1,g_2,...,g_d) \in \mathbb{R}^d$ and then define $g:\mathcal{P}(X)\to\mathbb{R}$ by $g(p) = \sum_{j=1}^{d} g_j p_j$. In this way we get
$$\ell_h (g)= \bigoplus_{p \in \mathcal{P}(X)} h(p) \odot  g(p)= \sup_{p\in \mathcal{P}(X)}  \{ (- \sum_j p_j \log p_j ) +  \sum_j g_j p_j\}. $$
Question: given such $g$,  what is the probability $p\in \mathcal{P}(X)$ which  attains the value  $\ell_h (g)$? It is well known that the  Gibbs probability, $p_j= \frac{e^{g_j}}{ \sum_k e^{g_k}}$, $j\in\{1,2,...,d\}$, is the solution  (see Lemma 9.9 in \cite{Walters}).

We consider now the level-2 case. In this way we fix a continuous function $g:\mathcal{P}(X)\to\mathbb{R}$ and the idempotent pressure of $g$ which is given by
\[\ell_h (g)= \bigoplus_{p \in \mathcal{P}(X)} h(p) \odot  g(p)= \sup_{p\in \mathcal{P}(X)}  \{ (- \sum_j p_j \log p_j ) +  g(p_1,...,p_d)\}.\]
Clearly we can not claim that there exist a unique equilibrium measure. Denoting $f(p_1,...,p_d) := (- \sum_j p_j \log p_j ) +  g(p_1,...,p_d)$ we are just maximizing this continuous function $f$ over the compact set $\mathcal{P}(X)$. Depending of the nature of $g$ (and then of $f$) we can have a different subset of $\mathcal{P}(X)$ as the set of equilibrium measures. It does not need to be a convex subset, but just a closed subset of $\mathcal{P}(X)$.

\end{example}

\begin{example}\label{remark -infty else}
Our formulation encompasses several classical constructions. For example, if $T: X \to X$ is a measurable map and $\mathcal{M}(T)$ is the set of invariant probabilities, by taking  $h=0$ over $\mathcal{M}(T)$ and $-\infty$ else, the idempotent pressure became
$\ell(g)= \max_{\mu \in \mathcal{M}(T)}  g(\mu), $
which extends the standard level-1 ergodic optimization problem $\max_{\mu\in \mathcal{M}(T)}  \int_{X} A(x) d\mu(x)$. If $h$ is the Kolmogorov-Sinai entropy over $\mathcal{M}(T)$ and $-\infty$ else, the idempotent pressure is given by
$\ell(g)= \max_{\mu \in \mathcal{M}(T)}  [g(\mu)+h(\mu)], $ which in level-1 is given by $$\max_{\mu\in \mathcal{M}(T)}  \int_{X} A(x) d\mu(x) +h(\mu),$$
usually known as the variational principle for pressure.
\end{example}

Now we consider level-1 to level-2  variational principles where the dependence on a potential function, acting in measures, is nonlinear.  This formalism includes the study of the Curie-Weiss model, which is of great importance in Statistical Mechanics (see Section 2 in \cite{FV} and also \cite{Barre}, \cite{BKL}, \cite{LeWa1}, \cite{LeWa2}).

\begin{example} \label{NTF}
We consider  on $\Omega:=\{1,2,...,d\}^\mathbb{N} $ the distance $d(x,y)=2^{-N}$, where $N$ is the smallest natural number  such that $x_j\neq y_j$, where $x=(x_1,x_2,x_3,...)$, $ y= (y_1,y_2,y_3,...).$ Consider also the shift map $\sigma:  \Omega \to \Omega$, given by $\sigma(x)=(x_2,x_3,x_4,...)$. We denote by $\mathcal{M}(\sigma)$ the set of $\sigma$-invariant probabilities on the Borel sigma-algebra of  $\Omega$. Consider a continuous function (a potential) $A:\Omega \to \mathbb{R}$. Given a continuous function $g:\mathcal{P}(\Omega) \to \mathbb{R}$, consider the functional $\mathfrak{l}_h(g):= \sup_{ \mu \in \mathcal{P}(\Omega)} h(\mu) + g(\mu)$, where $h$ is the \textbf{extended entropy} (it coincides with the Kolmogorov-Sinai entropy in the set of invariant probabilities and it is $-\infty$ for non-invariant probabilities). Then, we can define the ``quadratic" pressure for $A$ as
$$P_2(A)= \sup_{\mu \in \mathcal{P}(\Omega)}\left\{ h(\mu) +  \left(\int_{X}A d\mu\right)^2 \right\}= \sup_{\mu \in \mathcal{M}(\sigma)}\left\{ h(\mu) +  \left(\int_{X}A d\mu\right)^2\right\}.$$

We call  $P_2(A)$ the   quadratic non-linear pressure for $A$ (in the sense of \cite{BKL}).  There exists a probability $\mathfrak{m}=\mathfrak{m}_{A,2}$ attaining the supremum value $P_2(A) ,$ where
	$\mathfrak{m}\in\mathcal{M}(\sigma)\subset  \mathcal{P}(\Omega)$. Observe that, as $\mu \to (\int A d \mu)^2$ is not affine, the probability maximizing $P_2(A)$ is not  unique and not necessarily ergodic, see \cite{BKL}.

\end{example}

\medskip

\begin{definition} \label{eenf} Given a continuous function $F: \mathbb{R} \to \mathbb{R}$ and a continuous potential $A: \Omega \to \mathbb{R}$, a probability $\mu=\mu_{F,A}$ is called the  $(F,A)$-equilibrium probability, if it maximizes the nonlinear pressure
$$P_F(A)= \sup_{\mu \in\mathcal{P}(\Omega)} \{ h(\mu) + F\left(\int_{X}A d\mu\right) \}= \sup_{\mu \in \mathcal{M}(\sigma)} \{h(\mu) + F\left(\int A d\mu\right)\},$$
where $h$ is the extended entropy.
\end{definition}

The classical pressure is obtained when $F$ is the identity.

\section{Invariant idempotent pressures}\label{sec: invariant}
We denote by $I(\mathcal{P}(X))$ the set of idempotent pressures functions (even using this notation we do not assume $\ell(0)=0$). Given a continuous map $S: I(\mathcal{P}(X)) \to I(\mathcal{P}(X))$, we will say that an idempotent pressure function $\ell \in I(\mathcal{P}(X))$ is $S$-invariant if  $S(\ell)=\ell$.  When $T: X \to X$ is a continuous dynamical system we can use the  functorial action of $T$ in $C(X, \mathbb{R})$ to define, by duality, a map in $\mathcal{P}(X)$, which is continuous for the weak$^*$ topology. It is the push-forward map $T^{\sharp}:  \mathcal{P}(X) \to \mathcal{P}(X)$, where $ T^{\sharp}(\mu)=\nu$ means that $\int f\,d\nu = \int f\circ T\,d\mu,\,\forall f\in C(X, \mathbb{R}). $

The usual pressure in thermodynamic formalism for level-1 functions has the property $P(f\circ T) = P(f)$, for any $f\in C(X,\mathbb{R})$. In this way it is natural to ask what are the idempotent pressures functions for level-2 functions satisfying
$\ell(g) =\ell(g\circ T^{\sharp})$ for any $g\in C(\mathcal{P}(X),\mathbb{R})$.

Our first result is the following
\begin{proposition}\label{prop: T sharp}
  Let $(X,d)$ be a compact metric space and $T:X \to  X$ be a continuous map. Let $S:I(\mathcal{P}(X)) \to I(\mathcal{P}(X))$ be the map given by
$$S(\ell)(g):=\ell(g\circ T^{\sharp}), \; \ell \in I(\mathcal{P}(X)).$$
If $h \in U(\mathcal{P}(X))$ satisfies $h(\mu) = -\infty$ if $\mu$ is not $T-invariant$, then the associated idempotent pressure is invariant for $S$. Furthermore, $\ell$ is invariant for $S$ if, and only if, its density entropy satisfies
  $$h_{\ell}(\nu)=\left\{\begin{array}{cl} \sup_{[\mu \in \mathcal{P}(X)\; |\; T^{\sharp}(\mu)=\nu]} h_{\ell}(\mu), &\; \text{if}\, \,\nu \in T^\sharp(\mathcal{P}(X))\\ -\infty & \text{ if} \,\,\nu\notin T^\sharp(\mathcal{P}(X)) \end{array}\right..$$
\end{proposition}

\begin{proof}
	
It suppose initially that $h$ is a u.s.c. function satisfying $h(\mu) = -\infty$ if $\mu$ is not $T$-invariant. Define $\ell(g) := \sup_{\mu\in P(X)}g(\mu)+h(\mu)$. We have then
	\[S(\ell)(g) = \ell(g\circ T^{\sharp} ) = \sup_{\mu\in P(X)}g(T^{\sharp}(\mu))+h(\mu). \]
	If $T^{\sharp}(\mu)\neq \mu$ then $h(\mu)=-\infty$ and consequently such $\mu$ does not attain the supremum. It follows that the above supremum is equal to  $\sup_{\mu\in P(X)}g(\mu)+h(\mu)$ which is also equal to $\ell(g)$. This proves that $\ell$ is invariant for $S$.
	
Now we consider an idempotent pressure $\ell$.  Let us denote by $h_{\ell}$ its unique u.s.c. density entropy.  So we have
$\ell(g) = \sup_{\mu \in \mathcal{P}(X)} h_{\ell}(\mu) + g(\mu)$. We start by studying $S(\ell)$.
\[S(\ell)(g) = \ell(g\circ T^{\sharp}) = \sup_{\mu \in \mathcal{P}(X)} h_{\ell}(\mu) + g(T^{\sharp}(\mu)) = \sup_{\nu \in \mathcal{P}(X)} h'(\nu) + g(\nu)\]
where
\[h'(\nu)=\left\{\begin{array}{cl} \sup_{[\mu \in \mathcal{P}(X)\; |\; T^{\sharp}(\mu)=\nu]} h_{\ell}(\mu), &\; \text{if}\, \,\nu \in T^\sharp(\mathcal{P}(X))\\ -\infty & \text{ if} \,\,\nu\notin T^\sharp(\mathcal{P(X)}) \end{array}\right.\]
We claim that $h' \in U(\mathcal{P}(X))$. Indeed, if $h_\ell(\mu)\neq -\infty$ and $T^\sharp(\mu)=\nu$ we have $h'(\nu)\neq -\infty$ and then $\supp(h')\neq \emptyset$. As $T^\sharp$ is continuous, the set $[\mu \in \mathcal{P}(X)\; |\; T^{\sharp}(\mu)=\nu]$ is compact for any $\nu$ and then, if $\nu \in T^\sharp(\mathcal{P}(X))$ there exists $\mu \in (T^\sharp)^{-1}(\nu)$ such that $h_{\ell}(\mu)=h'(\nu)$ (it can be $-\infty$).  In order to show that $h'$ is u.s.c. consider that $\nu_n\to\nu$ in the weak$^*$ topology. We need to show that $\limsup h'(\nu_n)\leq h'(\nu)$. So we can suppose $h'(\nu_n)\neq -\infty$ for all $n$. Let $(\mu_n)$ be such that $T^\sharp(\mu_n)=\nu_n$ and $h_\ell(\mu_n)=h'(\nu_n)$. As the set $\mathcal{P}(X)$ is compact, by taking a subsequence, we can suppose there is $\mu$ such that $\mu_n \to \mu$. By continuity of $T^{\sharp}$ we get $T^\sharp(\mu) = \nu$. Using that $h_\ell$ is u.s.c, we have
\[h'(\nu) \geq h_\ell(\mu) \geq \lim_{n} h_\ell(\mu_n) = \lim_n h'(\nu_n),\]
as claimed.

As $S(\ell)$ has density $h'\in U(\mathcal{P}(X))$ and $\ell$   has density $h_\ell\in U(\mathcal{P}(X))$ from the uniqueness of the density, see Theorem \ref{thm: represetation variation}, we get $S(\ell)=\ell$ iff
$$h_{\ell}(\nu)=\left\{\begin{array}{cl} \sup_{[\mu \in \mathcal{P}(X)\; |\; T^{\sharp}(\mu)=\nu]} h_{\ell}(\mu), &\; \text{if}\, \,\nu \in T^\sharp(\mathcal{P}(X))\\ -\infty, & \text{ if} \,\,\nu\notin T^\sharp(\mathcal{P(X)}) \end{array}\right..$$
\end{proof}

\begin{example} In above proposition, $h$ does not need to be $-\infty$ over non-invariant probabilities in order to get $S(\ell)=\ell$. For example, consider $X=\{
1,2\}$. Then $\mathcal{P}(X)=\{(p,1-p),\, p\in[0,1]\}$. Consider the Shannon entropy $h(p,1-p) = -p\log(p)-(1-p)\log(1-p)$, which satisfies $0\leq h(p,1-p) \leq \log(2)$ for any $p\in[0,1]$. Suppose that $T:\{1,2\}\to\{1,2\}$ satisfies $T(1)=2$ and $T(2)=1$. Then $T^{\sharp}(p,1-p) = (1-p,p)$ and the unique $T-$invariant probability is $(\frac{1}{2},\frac{1}{2})$. Observe, however, that $h(T^\sharp(\mu)) = h(\mu)$ for any $\mu\in\mathcal{P}(X)$, that is $h(p,1-p)=h(1-p,p)$. Consequently if $\ell$ has density entropy $h$ we have
\[\ell(g\circ T^\sharp) = \sup_{\mu\in\mathcal{P}(X)}g(T^\sharp(\mu))+h(\mu) =  \sup_{\mu\in\mathcal{P}(X)}g(T^\sharp(\mu))+h(T^\sharp(\mu)) = \ell(g).\]
In this example
 $$h(\nu)= \sup_{[\mu \in \mathcal{P}(X)\; |\; T^{\sharp}(\mu)=\nu]} h(\mu)$$
corresponds to $h(p,1-p)= h(1-p,p)$.
\end{example}

Now we study a backwards  idempotent pressure function of the previous theorem.

\begin{example} Let $(X,d)$ be a compact metric space and $T:X \to  X$ be a surjective and continuous map such that any point as a finite number of pre-images. Suppose that $L_J:C(X, \mathbb{R}) \to C(X, \mathbb{R})$ is a transfer operator associated to a continuous Jacobian $J:X\to \mathbb{R}$ and consider the dual map $L_J^{*}:  \mathcal{P}(X) \to \mathcal{P}(X)$.  It means that
$$L_J(f)(x)=\sum_{T(y)=x} J(y) f(y), \; f \in C(X, \mathbb{R}),$$
where $\sum_{T(y)=x} J(y)=1,\, \forall x\in X$. Furthermore, $L_J^*(\mu)=\nu$ means that $\int f\,d\nu = \int L_J(f)\,d\mu$.
By Tichonoff-Schauder theorem, there is a probability $\mu_J$ satisfying $L_J^*(\mu_J)=\mu_J$.

We claim that
$T^{\sharp}\circ L_J^{*}(\mu)=\mu, \; \forall \mu \in \mathcal{P}(X)$ (see Lemma 2.4 in \cite{LO1}). Indeed, if $L_J^*(\mu)=\nu$ and $T^\sharp(\nu)=\omega$ then we  have, for any $f\in C(X, \mathbb{R})$,
\[\int f\,d\omega = \int f\circ T \,d\nu = \int \sum_{T(y)=x}J(y)f(T(y))\,d\mu(x) \]\[= \int \sum_{T(y)=x}J(y)f(x)\,d\mu(x)=\int f(x)\,d\mu(x).\]
\end{example}

\begin{proposition}
 Under the setting of above example, consider the map $S:I(\mathcal{P}(X)) \to I(\mathcal{P}(X))$, given by
  $$S(\ell)(g):=\ell(g\circ L_J^{*}) .$$
  If $h$ is a u.s.c function satisfying $h=-\infty$ at any probability which is not invariant for $L_J^*$, then its associated idempotent pressure function is invariant for $S$. More generally, the following sentences concerning an idempotent pressure $\ell$ and its density entropy $h_\ell$ are equivalent:\newline
  i. $\ell$ is invariant for $S$;\newline
  ii. $h_\ell$ satisfies $h_\ell = h_\ell\circ T^\sharp$ and $h_\ell(\nu)=-\infty$ if $\nu\notin L_J^*(\mathcal{P}(X))$;\newline
  iii. $h_\ell$ satisfies $h_\ell = h_\ell\circ L_J^*$ and $h_\ell(\nu)=-\infty$ if $\nu\notin L_J^*(\mathcal{P}(X))$.
\end{proposition}
\begin{proof}
	The first part of the proof follows the same lines than the proof of Proposition \ref{prop: T sharp}, replacing $T^\sharp$ by $L_J^*$ in that proof.

  Let us write $$\ell(g)= \max_{\mu \in \mathcal{P}(X)} h_{\ell}(\mu) + g(\mu)$$ where $h_\ell$ is u.s.c.

  (iii. $\Rightarrow$ i.) Suppose that $h_{\ell}(\mu)=h_\ell(L_J^*(\mu))$ for any $\mu\in \mathcal{P}(X)$ and $h_\ell(\nu)=-\infty$ if $\nu \notin  L_J^*(\mathcal{P}(X))$.
  We have
   \[S(\ell)(g) = \ell(g\circ L_J^{*}) = \max_{\mu \in \mathcal{P}(X)} h_{\ell}(\mu) + g(L_J^*(\mu))=\max_{\mu \in \mathcal{P}(X)} h_{\ell}(L_J^*(\mu)) + g(L_J^*(\mu))\]\[=\max_{\nu \in L_J^*(\mathcal{P}(X))} h_{\ell}(\nu) + g(\nu)=\max_{\nu \in \mathcal{P}(X)} h_{\ell}(\nu) + g(\nu),\]
   where in last equality we use that $h_\ell(\nu)=-\infty$ if $\nu \notin  L_J^*(\mathcal{P}(X))$ and then such $\nu$ does not attain the maximum. This proves that $\ell$ is $S$-invariant.

   (i. $\Rightarrow$ ii.) Suppose that  $\ell(g)=\ell(g\circ L_J^{*})$, for any $g\in C(\mathcal{P}(X),\mathbb{R})$.
 Observe that
   $$\ell(g\circ L_J^{*})   =\max_{\mu \in \mathcal{P}(X)} h_{\ell}(\mu) + g(L_J^{*}(\mu))$$ $$ =\max_{\mu \in \mathcal{P}(X)} h_{\ell}(T^{\sharp}(L_J^{*}(\mu)) + g(L_J^{*}(\mu)) =  \max_{\nu \in \mathcal{P}(X)} h'(\nu) + g(\nu),$$
where
\[h'(\nu):=
\left\{
  \begin{array}{ll}
    h_{\ell}(T^{\sharp}(\nu)), &  \nu \in L_J^{*}(\mathcal{P}(X)) \\
    -\infty, & \text{ otherwise.}
  \end{array}
\right.
\]
Since $L_J^{*}(\mathcal{P}(X))$ is closed we obtain that $h'$ is u.s.c. Thus $h'$ is the density of $S(\ell)$.
Since $S(\ell)=\ell$, from the uniqueness of the density $h_{\ell}$, we conclude that $h'=h_{\ell}$.

(ii. $\Rightarrow$ iii.) For any $\mu$ we have $h_\ell(L_J^*(\mu)) = h_\ell(T^{\sharp}(L_J^*(\mu)))= h_{\ell}(\mu)$.
\end{proof}

Next theorem consider a generalization of above one for the case of multiple Jacobians in the context of symbolic dynamics. Observe that for a unique Lipschitz Jacobian in symbolic dynamics, if we iterate the above process, we get that $h_{\ell}(\mu) \neq -\infty$ just if $\mu \in (L_J^{*})^n(\mathcal{P}(X)) $ for any $n$. There exists a unique such $\mu$ and it is also the unique invariant probability for $L_J^*$ (eigen-probablity).

We consider  on $\Omega=\{1,2,...,d\}^\mathbb{N} $ the distance $d_\gamma$ defined by \begin{equation} \label{alfa}d_\gamma( x, y) = \gamma^{ i(x,y)},
\end{equation}
where $0<\gamma <\frac{1}{d+1},$ $x =(x_0,x_1,x_2,...), y=(y_0,y_1,y_2,...)$ and $i(x,y)= $ min $\{j\in \mathbb{N} , x_j \neq y_j\}$. In this case the diameter of $\Omega$ is equal to $1$. A cylinder is a subset of $\Omega$ in the form
\[[x_1,...,x_n] = \{(y_1,y_2,y_3,...)\in\Omega\,|,y_1=x_1,y_2=x_2,...,y_n=x_n\}.\]

\begin{theorem}\label{thm: IFS invariant variation on shift}
  Consider the full shift $\sigma$ over the space $(\Omega,d_\gamma)$, where $\Omega=\{1,...,d\}^{\mathbb{N}}$ and $0<\gamma <\frac{1}{d+1}$. Let $$\mathcal{D}=\{J:\Omega \to [0,1]\,|\,\operatorname{Lip}(J)\leq 1\, \,\text{and}\, \,\sum_{a=1}^{d}J(ax)=1\,\forall x\in \Omega\}.$$  Consider, for each $J\in \mathcal{D}$, the dual map $L^{*}_{J}:  \mathcal{P}(X) \to \mathcal{P}(X)$.\newline
  1.  For each sequence $(J_1,J_2,J_3...)$ of elements of $\mathcal{D}$ there exists a unique probability $\mu\in \mathcal{P}(X)$ such that  $$\mu = \lim_{n \to \infty} L^{*}_{J_1}\circ\dots\circ L^{*}_{J_n}(\nu ),$$ for any probability $\nu$. \newline
2.   Consider $\mathcal{D}$ as a metric space with the supremum norm and for a fixed closed subset $D\subseteq \mathcal{D}$ let $q: D \to \mathbb{R}$ be a continuous function, such that, $\sup_{J \in D} q_{J} =0$.
  Consider the operator $\mathcal{M}: I(P(X)) \to I(P(X)),$ acting over idempotent pressures, which is given by
 \begin{equation} \label{prew}	\mathcal{M}(\ell)(f):= \bigoplus_{J \in \mathcal{D}} q_{J}\odot\ell( f\circ L^{*}_{J}).
 \end{equation}
 There exists a unique idempotent pressure function $\ell$ satisfying $\ell(0)=0$ and invariant for $\mathcal{M}$. Its density entropy is given by
  $$h_{\ell}(\mu)= \bigoplus_{\begin{array}{c}(J_1,J_2,J_3,...)\in D^{\mathbb{N}}\, such \, that\\ L^{*}_{J_1}\circ\dots\circ L^{*}_{J_n}\to \mu\end{array}  }[q_{J_1}+q_{J_2}+q_{J_3}+...]$$
  and $h_{\ell}(\mu)=-\infty$ if there is not such sequence $(J_1,J_2,J_3,...)$.
\end{theorem}
Equation \eqref{prew} can be seen as an IFS idempotent form of the dual of the Ruelle operator.

\begin{example} Suppose that $J_1$ and $J_2$ are two Jacobians in the set $\mathcal{D}$. in this way $D=\{J_1,J_2\}$. Suppose that $q_{J_1}=q_{J_2}=0$.
	 Then we are considering the operator $\mathcal{M}$, acting over idempotent pressures, which satisfies
		\[		\mathcal{M}(\ell)(f):= \ell( f\circ L^{*}_{J_1})\oplus \ell( f\circ L^{*}_{J_2}) .\]
		There exists a unique idempotent pressure function $\ell$ satisfying $\ell(0)=0$ and invariant for $\mathcal{M}$. Its density entropy is given by $h_\ell(\mu)=0$ if there is a sequence $(i_1,i_2,i_3,...)\in\{0,1\}^{\mathbb{N}}$ such that $\mu = \lim_{n \to \infty} L^{*}_{J_{i_1}}\circ\dots\circ L^{*}_{J_{i_n}}(\cdot)$ and $h_\ell(\mu)=-\infty$ if there is not such a sequence. 	
\end{example}

\begin{example} Suppose $\Omega=\{1,2\}^{\mathbb{N}}$ and for each $p\in[0,1]$ let $J_p$ be defined by $J_p(1,x_1,x_2,...)=p$ and $J_p(2,x_1,x_2,...)=1-p$. Let $D=\{J_p\,|\,p\in [0,1]\}\subset \mathcal{D}$. We denote also $P_p(1)=J_p(1)=p$ and $P_p(2)=J_p(2)=1-p.$
Fix a sequence $(J_{p_1},J_{p_2},J_{p_3},...)$ and let $\mu$ be such that
 $$L^{*}_{J_{p_1}}\circ\dots\circ L^{*}_{J_{p_n}}(\cdot ) \to \mu.$$
We want to characterize such $\mu$. Let us claim that $\mu$ satisfies $$\mu[x_1,...,x_n] = P_{p_1}(x_1)\cdots P_{p_n}(x_n),$$ for any cylinder set $[x_1,...,x_n]$ (this is a generalization of the concept of Bernoulli probabiliy, because the weight $(p_i,1-p_i)$ is not the same for each coordinate $i$).  In order to verify that $\mu$ satisfies such property we compute as below.\newline
- if $f_1=I_{[x_1]}$ we have, for any $\nu$:
\[\int f_1 d L^{*}_{J_{p_1}}(\nu) = \int L_{J_{p_1}}(f_1)\,d\nu = \int J_{p_1}(1x)f_1(1x)+J_{p_1}(2x)f_1(2x) \,d\nu (x) \]
\[= p_1I_{[x_1]}(1) + (1-p_1)I_{[x_1]}(2)=P_{p_1}(x_1).\]
It follows that $\mu[x_1] = P_{p_1}(x_1)$.\newline
- if $f_2=I_{[x_1,x_2]}$ we have, for any $\nu$:
\[\int f_2\, d (L^{*}_{J_{p_1}}\circ L^{*}_{J_{p_2}} (\nu)) = \int L_{J_{p_1}}(f_2)\,dL^{*}_{J_{p_2}}\,d\nu = \int L_{J_{p_2}}\circ L_{J_{p_1}}(f_2)\,d\nu\]
\[=\int \sum_{i_2=1}^{2} J_{p_2}(i_2)(L_{J_{p_1}}f_2)(i_2x) \,d\nu (x) =\int \sum_{i_2=1}^{2} J_{p_2}(i_2)(\sum_{i_1=1}^{2}J_{p_1}(i_1)f_2(i_1i_2x)) \,d\nu (x)\]\[=\sum_{i_2=1}^{2} J_{p_2}(i_2)(\sum_{i_1=1}^{2}J_{p_1}(i_1)f_2(i_1i_2))
=\sum_{i_1,i_2}J_{p_1}(i_1)J_{p_2}(i_2)I_{[x_1,x_2]}(i_1,i_2)     =P_{p_1}(x_1)\cdot P_{p_2}(x_2).\]
In general, if $f_n=I_{[x_1,x_2,...,x_n]}$ we have, for any $\nu$:
 \[\int f_n \,d L^{*}_{J_{p_1}}\circ\cdots \circ L^{*}_{J_{p_n}} (\nu) =  \int L_{J_{p_n}}\circ\cdots \circ L_{J_{p_1}}(f_n)\,d\nu\]
 \[=\int \sum_{i_1,...,i_n} J_{p_n}(i_n)\cdots J_{p_1}(i_1)I_{[x_1,...,x_n]}(i_1,...,i_n) =P_{p_1}(x_1)\cdots P_{p_n}(x_n).\]

Such $\mu$ usually is not $\sigma$-invariant. Indeed, a probability $\nu$ is invariant if it satisfies, for any cylinder set $[x_1,...,x_n]$,
\[\sum_{i=1}^{2}\nu[i,x_1,x_2,...,x_n] = \nu[x_1,x_2,...,x_n].\]
For the above defined $\mu$ we have
\[\sum_{i=1}^{2}\mu[i,x_1,x_2,...,x_n] =\sum_{i=1}^{2}P_{p_1}(i)P_{p_2}(x_1)\cdots P_{p_{n+1}}(x_n) = P_{p_2}(x_1)\cdots P_{p_{n+1}}(x_n)\]
which is different of $\mu[x_1,x_2,...,x_n]=P_{p_1}(x_1)\cdots P_{p_{n}}(x_n)$. We get the equality for all cylinders just in the case $p_1=p_2=p_3=...$, that is $\mu$ is the Bernoulli probability.
\end{example}

\subsection{Proof of Theorem \ref{thm: IFS invariant variation on shift}}
In order to prove Theorem~\ref{thm: IFS invariant variation on shift} we need to consider several concepts  and prove some technical preparatory results. In particular, we need some stability result showing how the image of a transfer operator acting in probabilities change under small variations of the parameter. The proof follows from adapting results described in  \cite{Hut} and \cite{Micu} to our level-2 context.

For a Lipschitz functions $w:\Omega\to\mathbb{R}$, denote $|w|_{\gamma}= \sup_{x \neq y} \frac{|w(x)- w(y)|}{d_\gamma(x,y)},$
which is called the Lipschitz  constant of $w$, and
$||w||_\infty=\sup_{x\in \Omega}|w(x)|,$
which is the supremum norm of $w$.

Consider the set  $$\mathcal{D}=\{J:\Omega \to [0,1]\,|\,|J|_\gamma\leq 1\, \,\text{and}\, \,\sum_{a=1}^dJ(ax)=1\,\forall x\in \Omega\}.$$ It is a compact metric space for the supremum norm $||\cdot||_\infty$. Indeed,
as the potentials $J\in\mathcal{D}$ are uniformly bounded and have Lispchitz constant smaller or equal than $1$, any sequence has a convergent subsequence  by Arzela-Ascoli Theorem. Furthermore the limit function is in  $\mathcal{D}$.

\begin{definition} We define the $1$-Wasserstein distance on $\mathcal{P}(\Omega)$ by
	\begin{equation} \label{WW11} W_{1}(\mu,\nu)= \sup_{|f|_{\gamma}\leq 1} \mu(f)-\nu(f).
	\end{equation}
\end{definition}
Observe that for any constant $c$ we get $|f+c|_{\gamma}=|f|_\gamma$ and $\mu(f+c)-\nu(f+c) = \mu(f)-\nu(f)$. Then we can suppose that $\inf\{f(x)\,|\,x\in\Omega\}=0$. Consequently, as $diam(\Omega)=1$, we get also $\sup\{f(x)\,|\,x\in\Omega\}\leq 1$. Consequently we can suppose $0\leq f\leq 1$ for the computation of $W_1$.

\begin{proposition} \label{eewx1} For any $J\in\mathcal{D}$ and any $\mu,\nu \in \mathcal{P} (\Omega)$,
	\begin{equation} \label{ewq123}W_{1}(L_{J}^* ( \mu), L_{J}^* ( \nu))\leq\,r\cdot  W_{1}( \mu,  \nu),
	\end{equation}
where $r=(d+1)\gamma<1$.	
\end{proposition}

This shows that $\mu \to L_{J}^* (\mu)$ is a contraction on $\mathcal{P} (\Omega)$ for the metric $W_{1}$, with a constant $r$ which is independent of $J \in\mathcal{D}$.

\begin{proof} 	We initially claim that $|L_{J}(f)|_{\gamma}\leq (d+1)\cdot\gamma\cdot |f|_\gamma$, for any Lipschitz function $f$ satisfying $\inf\{f(x)\,|\,x\in\Omega\}=0$ (and consequently $\sup\{f(x)\,|\,x\in\Omega\}\leq|f|_\gamma$).
	Indeed, for $x\neq y$ we have
	\[ \frac{|L_{J}(f)(x)-L_{J}(f)(y)|}{d_\gamma(x,y)} = \frac{|\sum_{a=1}^{d}J(ax)f(ax)-\sum_{a=1}^{d}J(ay)f(ay)|}{d_\gamma(x,y)} \]
	\[\leq \frac{|\sum_{a=1}^{d}J(ax)f(ax)-\sum_{a=1}^{d}J(ay)f(ax)| +|\sum_{a=1}^{d}J(ay)f(ax)-\sum_{a=1}^{d}J(ay)f(ay)|}{d_\gamma(x,y)} \]
	\[\leq \frac{|\sum_{a=1}^{d}[J(ax)-J(ay)]f(ax)| +|\sum_{a=1}^{d}J(ay)[f(ax)-f(ay)]|}{d_\gamma(x,y)} \]
	\[\leq d \gamma |J|_\gamma ||f||_\infty +\gamma |f|_\gamma
\leq d \gamma |f|_\gamma +\gamma |f|_\gamma = (d+1)\gamma|f|_\gamma .\]
Now for any Lipschitz function $f$ satisfying $|f|_\gamma \leq 1$ and $0\leq f\leq 1$, as $|L_J(f)|_\gamma \leq r$ where $r=(d+1)\cdot\gamma$,  we get
\[L_{J}^* ( \mu)(f) -  L_{J}^* ( \nu)(f) = \mu(L_J(f)) - \nu(L_J(f)) \leq \sup_{|g|_\gamma \leq r} \mu(g)-\nu(g)\]
\[ =\sup_{|h|_\gamma \leq 1} \mu(r\cdot h)-\nu(r\cdot h)=r\cdot W_{1}( \mu,  \nu). \]
Taking $\sup_{|f|_\gamma\leq 1}$ in the left hand side we conclude the proof.

	\end{proof}

\begin{proposition} \label{eewx2} For any probability $\mu\in \mathcal{P} (\Omega)$ and $J_1,J_2 \in \mathcal{D},$
	\begin{equation} \label{ewq12349}W_{1}({L}_{J_1}^* ( \mu), {L}_{J_2}^* (\mu))\leq d\, ||J_1- J_2||_\infty.
	\end{equation}
\end{proposition}

\begin{proof} Consider a function $f: \Omega \to [0,1]$ satisfying $|f|_{\gamma}\leq 1$.
We have
	$$| {L}_{J_1} (f) (x) - {L}_{J_2} (f) (x)|\leq \sum_{a=1}^d |f(ax) J_1(ax) - f(ax) J_2 (ax) |$$
	$$ \, \leq \sum_{a=1}^d  |J_1 (ax) - J_2(ax) | \leq  d\, ||J_1- J_2||_\infty.$$
Then,
	$$| \int f  d\, {L}_{J_1}^*(\mu) - \int f  d\, {L}_{J_2 }^*(\mu) |=
	|\int {L}_{J_1}  (f) (x) d \mu(x) -\int {L}_{J_2}  (f) (x) d \mu(x)   |\leq$$
	$$\int |\,{L}_{ J_1}  (f) (x) - {L}_{J_2}  (f) (x)\,| d \mu(x)    \leq  d\, ||J_{1}- J_{2}||_\infty. $$
	
\end{proof}

\begin{proposition} \label{eewx31}  Let $r=(d+1)\gamma <1$ and consider on $\mathcal{D}$ the metric defined by $\tilde d(J_1,J_2) = \frac{d}{r}||J_{1}- J_{2}||_\infty$. Then, for any $ J_1,J_2  \in  \mathcal{D}$ and any $\mu_1,\mu_2\in \mathcal{P} (\Omega)$ we have
	\begin{equation} \label{ewq12392} W_{1}({L}_{J_1}^* (\mu_1), {L}_{J_2}^* (\mu_2))\leq  r[ W_{1}( \mu_1,   \mu_2) + \tilde d(J_{1}, J_{2})]
	\end{equation}
\end{proposition}

\begin{proof} 	
	Note that from Propositions  \ref{eewx1} and \ref{eewx2}
	$$W_{1}({L}_{J_1}^* (\mu_1), {L}_{J_2}^* (\mu_2))\leq W_{1}({L}_{J_1}^* ( \mu_1), L_{J_1}^* (\mu_2)) + W_{1}({L}_{J_1}^* ( \mu_2), {L}_{J_2}^* (\mu_2))\leq $$
	$$ r\, \, W_{1}( \mu_1,   \mu_2)  \, +\, d\, ||J_{1}- J_{2}||_\infty = r[ W_{1}( \mu_1,   \mu_2) +\tilde d(J_{1}, J_{2})]. $$
\end{proof}

We remark that the metric $\tilde d$ on $\mathcal{D}$ is equivalent to the supremum norm.

\medskip

\begin{proposition} Given a closed (therefore compact) subset $D\subseteq \mathcal{D}$ with respect to the metric $\tilde d$ and the compact metric space $(\mathcal{P}(X),W_{1}),$ consider the iterated function system $(\mathcal{P}(X),(\phi_J)_{J\in D})$ where
$\phi_J: \mathcal{P}(X)\to \mathcal{P}(X)$ satisfies $\phi_J(\mu) = {L}_{J}^*  (\mu).$
Such IFS is uniformly contractible, that is, for any $J_1,J_2 \in D$ and $\mu_1,\mu_2 \in \mathcal{P}(X)$,
\[W_{1}(\phi_{J_1}(\mu_1), \phi_{J_2}(\mu_2))\leq  r[ W_{1}( \mu_1,   \mu_2) + \tilde d(J_{1}, J_{2})].\]
\end{proposition}

Applying above proposition, the proof of Theorem \ref{thm: IFS invariant variation on shift} is a consequence of Lemma 4.1 and Theorem 4.7 in \cite{MO1}.

\smallskip

\begin{remark} Once proved that $(\mathcal{P}(X),(\phi_J)_{J\in D})$ is uniformly contractible, Theorem 3.6 in \cite{MO1} can be also applied in order to get a characterization of invariant idempotent pressures for the non-place dependent case (where $q_J(\mu)$ depends of $\mu$).  Furhtermore, all theory concerning transfer operators for uniformly contractible IFS (see \cite{MO3}) can be applied for such IFS.
\end{remark}

\section{The inverse problem}\label{sec:The inverse problem}

We start investigating the inverse problem, that is: given an idempotent pressure function $\ell$, is there some uniformly contractive IFS in $I(\mathcal{P}(X))$ (recall from Section~\ref{sec: proof1}, Definition~\ref{def: Maslov measure}, that $I(\mathcal{P}(X))$ is the set of idempotent probabilities over $\mathcal{P}(X)$) for which $\ell$ is invariant?

An attempt to solve our inverse problem is to consider an IFS as general as possible. Let   $(X,d_X)$ be a compact metric space and  $(J := \mathcal{P}(X),d_J=\frac{1}{\gamma} d_{\mathcal{P}(X)})$, for $0<\gamma<1$, also a compact metric space. 	Consider the iterated function system given by a family of maps $\{\phi_{\nu}: \mathcal{P}(X)\to \mathcal{P}(X)\,|\,\nu \in \mathcal{P}(X)\}$ which are defined by
$$\phi_{\nu}(\mu)=\nu, \; \forall \nu \in \mathcal{P}(X) , \mu \in \mathcal{P}(X).$$
We notice that, $\phi$ is uniformly contractive, that is,
$$d_{\mathcal{P}(X)}(\phi_{\nu_1}(\mu_1),\phi_{\nu_2}(\mu_2))= d_{\mathcal{P}(X)}(\nu_1,\nu_2) + 0 \leq \gamma \cdot[d_J(\nu_1,\nu_2)+ d_{\mathcal{P}(X)}(\mu_1,\mu_2)].$$
 A continuous  function $q:\mathcal{P}(X) \times \mathcal{P}(X) \to \mathbb{R},\,q(\nu,\mu)=q_\nu(\mu),$ is a normalized family of weights if it satisfies $|q_\nu(\mu_1)-q_\nu(\mu_2)|<Cd(\mu_1,\mu_2)\,\forall \nu,\mu_1,\mu_2 \in \mathcal{P}(X)$ and $\oplus_{\nu\in \mathcal{P}(X) }q_\nu(\mu) = 0\,\forall \, \mu\in \mathcal{P}(X)$.  We call $\mathcal{S}=(\mathcal{P}(X), \phi, q)$ a max-plus IFS (mpIFS) (see \cite{MO1}).
\begin{definition}\label{def:mp ruelle operator}
	To each mpIFS $\mathcal{S}=(\mathcal{P}(X), \phi, q)$ we assign the following operators:\newline
	1.  $\mathcal{L}_{\phi,q}: C(\mathcal{P}(X), \mathbb{R})   \to C(\mathcal{P}(X), \mathbb{R})  $, defined by
	\begin{equation}\label{eq:mp ruelle operator}
		\mathcal{L}_{\phi,q}(f)(\mu):=\bigoplus_{\nu \in \mathcal{P}(X)} q_{\nu}(\mu)\odot f(\phi_{\nu}(\mu)).
	\end{equation}
	2.  $\mathcal{M}_{\phi,q}: I(\mathcal{P}(X)) \to I(\mathcal{P}(X)),$ defined by
	\begin{equation}\label{eq:Markov operator}
		\mathcal{M}_{\phi,q}(\ell):=\bigoplus_{\nu \in \mathcal{P}(X)}\ell( q_{\nu}\odot (f\circ \phi_{\nu})).
	\end{equation}
	3.  $L_{\phi,q}: U(\mathcal{P}(X)) \to U(\mathcal{P}(X)),$ defined by
	\begin{equation}\label{eq:transfer operator}
		L_{\phi,q}(\lambda)(\mu):=\bigoplus_{(\nu,\eta)\in \phi^{-1}(\mu)} q_{\nu}(\eta) \odot  \lambda(\eta).
	\end{equation}
\end{definition}
Next theorem establishes the relation between the three operators in Definition \ref{def:mp ruelle operator}.
\begin{theorem} \cite{MO1} \label{thm: equivalences} Given a density function $\lambda \in  U(\mathcal{P}(X))$ satisfying $\oplus_\mu \lambda(\mu) = 0$ and the associated idempotent pressure   $\ell=\bigoplus_{\mu\in \mathcal{P}(X)}\lambda(\mu)\odot\delta_\mu\in I(\mathcal{P}(X))$ we have that $\mathcal{M}_{\phi,q}(\ell)= \bigoplus_{\mu \in \mathcal{P}(X)} L_{\phi,q} (\lambda)(\mu) \odot \delta_\mu$, that is, $\mathcal{M}_{\phi,q}(\mu)$ has density $L_{\phi,q} (\lambda)$ where $\lambda$ is the density of $\ell$. Furthermore
	$$\mathcal{M}_{\phi,q}(\ell)(f) = \ell(\mathcal{L}_{\phi,q}(f)),$$
	for any $f \in C(X, \mathbb{R})$, that is, $\mathcal{M}_{\phi,q}$ is the max-plus dual of $\mathcal{L}_{\phi,q}$.
\end{theorem}
\begin{definition}\label{def:Markov operator}
	An idempotent pressure  $\ell \in I(\mathcal{P}(X))$ with density $\lambda \in  U(\mathcal{P}(X))$ is called invariant (with respect to the mpIFS) if it satisfies any of the following equivalent conditions:\newline
	1. $\mathcal{M}_{\phi,q}(\ell)=\ell$;\newline
	2. $L_{\phi,q}(\lambda)=\lambda$;\newline
	3. $\ell(\mathcal{L}_{\phi,q}(f))=\ell(f)$,  for any $f \in C(\mathcal{P}(X), \mathbb{R})$.
\end{definition}
{\bf Question:} Consider an idempotent pressure function $\ell$ satisfying $\ell(0)=0$ and with density entropy $h:=h_{\ell}$. Is  there a family of weights $q_{\nu}(\mu) \leq 0$, such that $\max_{\nu \in X} q_{\nu}(\mu) =0$, and $\ell= \max_{\mu \in \mathcal{P}(X)} h(\mu) + \delta_\mu$ is invariant for the max-plus IFS $(\mathcal{P}(X), \phi_\nu, q_{\nu})$?\\
The density entropy, as an invariant density, must verify $L_{\phi,q}(h)=h$, that is,
$$h(\mu)= \max_{(\nu,\eta)\in \phi^{-1}(\mu)} q_{\nu}(\eta) \odot  h(\eta).$$
Since $\phi_\nu(\mu)=\nu, \; \forall \nu \in \mathcal{P}(X) , \mu \in \mathcal{P}(X)$, we obtain $(\nu,\eta)\in \phi^{-1}(\mu)$ if, and only if $\nu=\mu$ and $\eta$ is arbitrary, thus,
\begin{equation}\label{eq: obstruction inverse problem}
  h(\mu)= \max_{\eta \in \mathcal{P}(X)} q_{\mu}(\eta) +  h(\eta).
\end{equation}
Note that $q_{\nu}(\zeta) = h(\nu) - h(\zeta)$ is not necessarily a candidate, because $h(\nu) - h(\zeta)$ is not necessarily bounded from above, neither continuous. Actually, for each $\nu$, we have that $q_{\nu}(\zeta)$ is a l.s.c. function.\\
Let us suppose that the entropy $h$ is continuous, non positive and that there exists $\mu_0 \in \mathcal{P}(X)$ such that $h(\mu_0)=0$. In this case, the solution is
$$q_{\mu}(\eta):=h(\mu), \; \forall \eta,\mu \in \mathcal{P}(X).$$
We can check that this function solves Equation~\ref{eq: obstruction inverse problem}.\\
Indeed,
$$h(\mu)= \max_{\eta \in \mathcal{P}(X)} q_{\mu}(\eta) +  h(\eta)= \max_{\eta \in \mathcal{P}(X)} h(\mu) +  h(\eta)$$
is always true because $h\leq 0$ and for $\eta=\mu_0$ we have $h(\eta)=0$.

\section{Max-plus dynamics}\label{sec:Max-plus dynamics}

Suppose $K$ is a compact  metric space, and  $T:K \to K$ is continuous.  Consider an ergodic $T$-invariant probability $\mu$ on $K$ and a continuous function $f:K \to \mathbb{R}$. It is natural to consider  the asymptotic, as $n \to \infty$,  of the sums
$$( f(x) \oplus f (T(x)) \oplus...\oplus f (T^{n-1} (x)),$$
for $\mu$-almost every $x\in K$. Among other things, we will introduce the max-plus partition function and we will describe large deviation properties.

The results of this section can be applied to the case where $T$ is the pushforward map and $K=\mathcal{P}(X)$, where $X$ is a compact metric space. A version of this result for uniquely ergodic  dynamical systems can be found in Corollary 4 in \cite{KM89}, which is the max-plus analogous to Furstenberg's theorem \cite{FURSTENBERG61}.

\smallskip

\begin{lemma}  \label{ler1} Suppose $K$ is a compact  metric space, and $T:K \to K$ is  continuous. Consider an ergodic $T$-invariant probability $\mu$ which is positive on open sets and a continuous function $f:K \to \mathbb{R}$. Then, for $\mu$-almost every $x\in K$ we have
\begin{equation} \label{out} \lim_{n \to \infty} f(x) \oplus (T(x)) \oplus...\oplus f (T^{n-1} (x)) = \sup_{x \in K} f(x).
\end{equation}
\end{lemma}

\begin{proof}
Let $x_0 \in K$ be such that $f(x_0)=\sup_{x \in K} f(x)$. Given $k\in\mathbb{N}$, let $\delta>0$ be such that $|f(x)- f(x_0)|<1/k$ for any $x\in  B(x_0,\delta)$.

As $\mu(B(x_0,\delta))>0$, by Birkhoff ergodic theorem, for $\mu$ almost every point $x\in K$, the sequence $x, T(x), T^2(x),...$ will visit the set $B(x_0,\delta)$. Therefore, there exists a set $U_k\subseteq K$ such that $\mu(U_k) = 1$ and for any $x\in U_k$,
$$\limsup_{n \to \infty} f(x) \oplus f (T(x)) \oplus...\oplus f (T^{n-1} (x)) \geq f(x_0)- 1/k.$$
Taking $U=\cap_{k\in\mathbb{N}} U_k$ we get $\mu(U) = 1$ and
$$\limsup_{n \to \infty} f(x) \oplus f (T(x)) \oplus...\oplus f (T^{n-1} (x)) \geq f(x_0)\,\,\,\forall x\in U.$$

\end{proof}

\bigskip

Let $K$ be a compact set, $T:K\to K$ be continuous, $f: K \to \mathbb{R}$ be a continuous function and $\mu$ be an ergodic $T$-invariant probability which is positive on open sets. For each value $t\in \mathbb{R}$, we define
$ c(t):=\limsup_{n\to \infty} c_{n} (t)$, where
	\begin{equation} \label{feij501}c_n(t):= \frac{1}{n} \log \int e^{ n \,t\, (f(x) \oplus f(T(x)) \oplus f(T^{2} (x)) \oplus ...
\oplus f(T^{n-1} (x) )) } d \mu (x).
\end{equation}

\begin{proposition}
For $t\geq 0$ we have $c(t) = t\cdot \max_{x\in K}f(x)$ and $c(-t) \geq (-t)\cdot  \max_{x\in K}f(x)$. 	
\end{proposition}

\begin{proof}
	Clearly, $c(0)=0$. Suppose $t>0$ and let $x_0\in K$ be such that $f(x_0)=\max_{x\in K}f(x)$. For each $\epsilon>0$ let $\delta>0$ be such that $f(x) >f(x_0)-\epsilon$ for any $x\in B(x_0,\delta)$. Then we have
	\begin{align*}
		c(t) &\geq \limsup_{n\to+\infty} \frac{1}{n}\log\int_{B(x_0,\delta)}e^{ntf(x)}d\mu(x)\\
		&\geq \limsup_{n\to+\infty} \frac{1}{n}\log\left(e^{nt(f(x_0)-\epsilon)}\cdot \mu(B(x_0,\delta))\right)\\
		&=t(f(x_0)-\epsilon) = t(\max_{x\in K}f(x)) - t\epsilon.
	\end{align*}
As $\epsilon$ is arbitrary, we get $c(t) \geq t\,\max_{x\in K}f(x)$. The opposite inequality is trivial. Furthermore,
 \begin{align*}
 	c(-t) &= \limsup_{n\to+\infty} \frac{1}{n}\log\int e^{- n \,t\, (f(x) \oplus  ...
 		\oplus f(T^{n-1} (x) ) } d \mu (x)\\
 	&\geq \limsup_{n\to+\infty} \frac{1}{n}\log\int e^{ -n \,t\, f(x_0) } d \mu =-tf(x_0)= -t(\max_{x\in K}f(x)).
 \end{align*}
\end{proof}
\medskip

\begin{remark} In Example \ref{ex: LDP} we exhibit a case where $c(-t) \neq (-t)\cdot  \max_{x\in K}f(x)$.
\end{remark}

\begin{remark} If $f\geq 0$ then $c$ is non-decreasing and consequently $c(s\oplus t)=c(s)\oplus c(t) $. If $f\geq 1$ and $\alpha <0$ then $c(\alpha \odot t) \leq \alpha \odot c(t)$. Indeed,
\[c(\alpha+t) \leq  \limsup_{n\to \infty} \frac{1}{n} \log \int e^{ (n\,\alpha \cdot 1 ) + n \,t\, (f(x) \oplus ...
		\oplus f(T^{n-1} (x) ) } d \mu (x) = \alpha +c(t).\]
Consequently, if $f\geq 1$ then $c$ is max-plus convex, that is, if $\alpha \oplus \beta = 0$ then $$c((\alpha \odot t)\oplus (\beta \odot s)) \leq (\alpha \odot c(t))	\oplus (\beta \odot c(s)).$$
	
	\end{remark}

In what follows we would like to estimate the growth with $n$ of the expression
\begin{equation} \label{try}   \int e^{n \,t\, (f(z) \oplus f(\sigma(z)) \oplus f(\sigma^{2} (z)) \oplus ...
\oplus f(\sigma^{n-1} (z) ) } d \mu (z),
\end{equation}
which is the max-plus version of the so called partition function.

\medskip

\begin{lemma}[ Chebyshev's inequality]
Let $g:K\to\mathbb{R}$ be a measurable function and $h:\mathbb{R} \to \mathbb{R}$ be a non-negative, non-decreasing function such that  $\int h(g(x)) d \mu (x)<+\infty$. Then, for any value $d$ such that $h(d)>0$,  $$ \mu (\{ x \mid g(x) \geq d \}) \leq \frac{ \int h(g(x)) d \mu (x)  }{ h(d)}.$$
\end{lemma}

In the classical case, the upper bound large deviation follows from Chebyshev's inequality (see \cite{Ellis}).
In the max-plus case we will also apply Chebyshev's inequality, but we  have to consider a small variation of it, taking into account our definition of  $c_n(t)$  (for the dynamical max-plus sum, in the exponential term in \eqref{try}, there is an extra term $n$ multiplying $t$, which does not appear in the case of the classical dynamical sum, as stated for $c(t)$ in page 535 in \cite{L3}.
\smallskip

\begin{proposition} \label{lape1} [Upper large deviation bounds] Let $K$ be a compact metric space, $T:K\to K$ be continuous, $f: K \to \mathbb{R}$ be a continuous function and $\mu$ be an ergodic $T$-invariant probability which is positive on open sets.  Then, for any value $b< \sup_{x\in K} f(x)$, we have
\begin{equation}\label{eq:ldp} \limsup_{n\to \infty} \frac{1}{n}\log (\mu\{x |  (f(x) \oplus ...
\oplus f(T^{n-1} (x) )\leq b \}) \leq \inf_{t\geq 0} [tb +c(-t)].\end{equation}

\end{proposition}
\begin{proof}
We can suppose $f\geq 0$. Indeed, if we add a constant on both $f$ and $b$ then both sides of \eqref{eq:ldp} remains equal.
For each $n\in\mathbb{N}$ and $t>0$, let us consider
 $$ h(x) = e^{n\, t\, x} , \, \, g(x) =-  (f(x) \oplus ...
\oplus f(T^{n-1} (x) ) \,\text{ and }  \, d=-b .$$
From Chebychev inequality, for $t\geq 0$
\begin{equation*}  \mu\{x |  - (f(x) \oplus  ...
	\oplus f(T^{n-1} (x) ) \geq -b\} \leq \frac{\int e^{-nt(f(x)  \oplus ...
			\oplus f(T^{n-1} (x) ))}\,d\mu(x)}{e^{-ntb}} .
\end{equation*}
Then
\[\limsup_{n\to+\infty}\frac{1}{n}\log\mu\{x |  - (f(x) \oplus  ...
\oplus f(T^{n-1} (x) ) \geq -b\} \leq c(-t) +tb.\]
Consequently
\[\limsup_{n\to+\infty}\frac{1}{n}\log\mu\{x |   (f(x) \oplus  ...
\oplus f(T^{n-1} (x) ) \leq b\} \leq \inf_{t\geq 0}[ c(-t) +tb].\]

\end{proof}

\begin{example}\label{ex: LDP} Let us suppose that $K=\{0,1\}^\mathbb{N}$ and $f:\{0,1\}^\mathbb{N}  \to \mathbb{R}$ is given by  $$f(x_1,x_2,x_3,...) =\left\{\begin{array}{cc} 0 & \text{if}\,\,\, x_1=0 \\ 1& \text{if}\,\,\, x_1=1\end{array}\right..$$ Let $T:\{0,1\}^\mathbb{N}\to\{0,1\}^\mathbb{N}$ be the shift map $\sigma$ and $\mu$ be the Bernoulli probability $(p,1-p)$, where $0<p<1$. This means that, for any cylinder set $[x_1,x_2,...,x_n]$, we have $\mu[x_1,x_2,...,x_n] = P_{x_1}\cdot P_{x_2}\cdots P_{x_n}$ where $P_{0}=p$ and $P_{1}=1-p$.
	
We remark that, for any $x\in\{0,1\}^\mathbb{N}$ we have that $(f(x) \oplus ...
\oplus f(\sigma^{n-1} (x) ))$ is equal to zero or one. Furthermore, it is zero iff $x\in [\underbrace{0,0,..,0}_{n}]$. Therefore, for any $0<b<1$ we have
\begin{align*}
	\lim_{n\to \infty} &\frac{1}{n}\log (\mu\{x |  (f(x) \oplus ...
	\oplus f(\sigma^{n-1} (x) ))\leq b \})\\ &= \lim_{n\to \infty} \frac{1}{n}\log (\mu[\underbrace{0,0,..,0}_{n}])=\lim_{n\to \infty}\frac{1}{n}\log(p^n)
	 =\log(p).\end{align*}

From now on we want to estimate the upper bound large deviation described in Proposition \ref{lape1} for $0<b<1$. As we will see, the right hand side of inequality \eqref{eq:ldp} does not match with $\log(p)$. Consequently, inequality \eqref{eq:ldp} can not be replaced by an equality.
In this way, initially we need to estimate $c(-t)$ for each $t\geq 0$.
For fixed $n\geq 0$ and $t\geq 0$ we have
\begin{align*}
	  \int &e^{n \,(-t)\, (f(x) \oplus f(\sigma(x)) \oplus  ...
			\oplus f(\sigma^{n-1} (x) ) } d \mu (x)\\
		&=  \int_{[1] } e^{- n \,t\, (f(x)  \oplus ...
		\oplus f(\sigma^{n-1} (x) ) } d \mu (x)\\
	&\,\,\,\,\,\,+\sum_{j=1}^{n-1} \int_{[\underbrace{0,0,..,0}_{j},\,1] } e^{- n \,t\, (f(x) \oplus ...
		\oplus f(\sigma^{n-1} (x) ) } d \mu (x)\\
	&\,\,\,\,\,\,+\int_{[\underbrace{0,0,..,0}_{n}] } e^{- n \,t\, (f(x) \oplus ...
		\oplus f(\sigma^{n-1} (x) ) } d \mu (x) \\
	&= e^{-nt}(1-p) +\sum_{j=1}^{n-1} e^{-n t}p^j(1-p) +  e^{ 0 }p^n \\
&=e^{- n t} (1-p)(\frac{1-p^n}{1-p})+ p^n\\
&= e^{- n t} (1-p^n)+ p^n .
\end{align*}
It follows that
\begin{align*}
c(-t)&=\limsup_{n\to+\infty} \frac{1}{n}\log(e^{- n t} (1-p^n)+ p^n)\\
&=\limsup_{n\to+\infty} \frac{1}{n}\log(e^{- n t} (1-p^n)+ e^{\log(p)n})=\max\{-t,\log(p)\}.\end{align*}
 Consequently, from inequality \eqref{eq:ldp} we get, for $0<b<1$,
\begin{align*}
	\limsup_{n\to \infty}& \frac{1}{n}\log (\mu\{x |  (f(x) \oplus ...
	\oplus f(\sigma^{n-1} (x) )\leq b \})\\
	& \leq \inf_{t\geq 0} [tb +\max\{-t,\log(p)\}]= \inf_{t\geq 0} [\max\{-t(1-b),tb+\log(p)\}] .
\end{align*}
The function $t\mapsto -t(1-b)$ is decreasing and the function $t\mapsto tb+\log(p)$ is increasing. Then, the above infimum is attained by $t$ satisfying $tb+\log(p) = -t(1-b)$, that is, $t=-\log(p)$. Therefore, inequality \eqref{eq:ldp} can be rewritten as
\[	\limsup_{n\to \infty} \frac{1}{n}\log (\mu\{x |  (f(x) \oplus ...
\oplus f(\sigma^{n-1} (x) )\leq b \}) \leq (1-b)\log(p).\]
On the other hand, as we saw above,
\[	\lim_{n\to \infty} \frac{1}{n}\log (\mu\{x |  (f(x) \oplus ...
\oplus f(\sigma^{n-1} (x) )\leq b \}) =\log(p),\]
being $\log(p)<(1-b)\log(p)$ because $\log(p)<0$.
\end{example}

\newcommand{\etalchar}[1]{$^{#1}$}


\begin{thebibliography}{BCM{\etalchar{+}}23}

\bibitem[BCM{\etalchar{+}}23]{BisCorrection}
Andrzej Bi{\'s}, Maria Carvalho, Miguel Mendes, Paulo Varandas, and Xingfu   Zhong.
\newblock Correction to: ``{A} convex analysis approach to entropy functions,
  variational principles and equilibrium states''.
\newblock {\em Commun. Math. Phys.}, 401(3):3335--3342, 2023.

\bibitem[BCMV22]{Bis}
Andrzej Bi{\'s}, Maria Carvalho, Miguel Mendes, and Paulo Varandas.
\newblock A convex analysis approach to entropy functions, variational
  principles and equilibrium states.
\newblock {\em Commun. Math. Phys.}, 394(1):215--256, 2022.

\bibitem[BH22]{Barre}
Luis Barreira and Carllos~Eduardo Holanda.
\newblock Higher-dimensional nonlinear thermodynamic formalism.
\newblock {\em J. Stat. Phys.}, 187(2):28, 2022.

\bibitem[BKL23]{BKL}
J{\'e}r{\^o}me Buzzi, Beno{\^{\i}}t Kloeckner, and Renaud Leplaideur.
\newblock Nonlinear thermodynamical formalism.
\newblock {\em Ann. Henri Lebesgue}, 6:1429--1477, 2023.

\bibitem[BLL13]{BLL}
Alexandre Baraviera, Renaud Leplaideur, and Artur Lopes.
\newblock {\em Ergodic optimization, zero temperature limits and the max-plus
  algebra. {Paper} from the 29th {Brazilian} mathematics colloquium --
  29{{\(^{\text o}\)}} {Col{\'o}quio} {Brasileiro} de {Matem{\'a}tica}, {Rio}
  de {Janeiro}, {Brazil}, {July} 22 -- {August} 2, 2013}.
\newblock Publ. Mat. IMPA. Rio de Janeiro: Instituto Nacional de Matem{\'a}tica
  Pura e Aplicada (IMPA), 2013.

\bibitem[BRZ10]{BRZ}
Lidia Bazylevych, Du{\v{s}}an Repov{\v{s}}, and Michael Zarichnyi.
\newblock Spaces of idempotent measures of compact metric spaces.
\newblock {\em Topology Appl.}, 157(1):136--144, 2010.

\bibitem[BS75]{Bauer}
Walter Bauer and Karl Sigmund.
\newblock Topological dynamics of transformations induced on the space of
  probability measures.
\newblock {\em Monatsh. Math.}, 79:81--92, 1975.

\bibitem[BV16]{BVerm}
Nilson C.~Jun. Bernardes and R{\^o}mulo~M. Vermersch.
\newblock On the dynamics of induced maps on the space of probability measures.
\newblock {\em Trans. Am. Math. Soc.}, 368(11):7703--7725, 2016.

\bibitem[DS58]{DF}
Nelson Dunford and Jacob~T. Schwartz.
\newblock Linear operators. {I}. {General} theory. ({With} the assistence of
  {William} {G}. {Bade} and {Robert} {G}. {Bartle}).
\newblock Pure and {Applied} {Mathematics}. {Vol}. 7. {New} {York} and
  {London}: {Interscience} {Publishers}. xiv, 858 p. (1958)., 1958.

\bibitem[Ell06]{Ellis}
Richard~S. Ellis.
\newblock {\em Entropy, large deviations, and statistical mechanics.}
\newblock Class. Math. Berlin: Springer, reprint of the 1985 edition edition,
  2006.

\bibitem[Fur61]{FURSTENBERG61}
H.~Furstenberg.
\newblock Strict ergodicity and transformation of the torus.
\newblock {\em Am. J. Math.}, 83:573--601, 1961.

\bibitem[FV18]{FV}
Sacha Friedli and Yvan Velenik.
\newblock {\em Statistical mechanics of lattice systems. {A} concrete
  mathematical introduction}.
\newblock Cambridge: Cambridge University Press, 2018.

\bibitem[Gar17]{G1}
Eduardo Garibaldi.
\newblock {\em Ergodic optimization in the expanding case. {Concepts}, tools
  and applications}.
\newblock SpringerBriefs Math. Cham: Springer, 2017.

\bibitem[Hut81]{Hut}
John~E. Hutchinson.
\newblock Fractals and self similarity.
\newblock {\em Indiana Univ. Math. J.}, 30:713--747, 1981.

\bibitem[KM89]{KM89}
V.~N. Kolokol'tsov and V.~P. Maslov.
\newblock Idempotent analysis as a tool of control theory and optimal
  synthesis. {I}.
\newblock {\em Funct. Anal. Appl.}, 23(1):1--11, 1989.

\bibitem[Kol01]{Kol01}
V.~N. Kolokol'tsov.
\newblock Idempotent structures in optimization.
\newblock {\em J. Math. Sci., New York}, 104(1):847--880, 2001.

\bibitem[LO24]{LO1}
A.~O. Lopes and E.~R. Oliveira.
\newblock Level-2 ifs thermodynamic formalism: Gibbs probabilities in the space
  of probabilities and the push-forward map.
\newblock {\em Dynamical Systems}, 0(0):1--25, 2024.

\bibitem[Lop90]{L3}
Artur~O. Lopes.
\newblock Entropy and large deviation.
\newblock {\em Nonlinearity}, 3(2):527--546, 1990.

\bibitem[LW19]{LeWa1}
Renaud Leplaideur and Fr{\'e}d{\'e}rique Watbled.
\newblock Generalized {Curie}-{Weiss} model and quadratic pressure in ergodic
  theory.
\newblock {\em Bull. Soc. Math. Fr.}, 147(2):197--219, 2019.

\bibitem[LW20]{LeWa2}
Renaud Leplaideur and Fr{\'e}d{\'e}rique Watbled.
\newblock Curie-{Weiss} type models for general spin spaces and quadratic
  pressure in ergodic theory.
\newblock {\em J. Stat. Phys.}, 181(1):263--292, 2020.

\bibitem[Mic14]{Micu}
Radu Miculescu.
\newblock Generalized iterated function systems with place dependent
  probabilities.
\newblock {\em Acta Appl. Math.}, 130(1):135--150, 2014.

\bibitem[MN22]{Mord}
Boris~S. Mordukhovich and Nguyen~Mau Nam.
\newblock {\em Convex analysis and beyond. {Volume} {I}. {Basic} theory}.
\newblock Springer Ser. Oper. Res. Financ. Eng. Cham: Springer, 2022.

\bibitem[MO24a]{MO1}
Jairo~K. Mengue and Elismar~R. Oliveira.
\newblock Invariant measures for place-dependent idempotent iterated function
  systems.
\newblock {\em J. Fixed Point Theory Appl.}, 26(2):35, 2024.

\bibitem[MO24b]{MO3}
Jairo.~K. Mengue and Elismar~R. Oliveira.
\newblock Large deviation for gibbs probabilities at zero temperature and
  invariant idempotent probabilities for iterated function systems, 2024.

\bibitem[Rod12]{Fagner}
Fagner~Bernardini Rodrigues.
\newblock {\em {Estudo das propriedades de algumas dinâmicas em
  $\mathcal{P}(X)$ : o push forward e a convolução}}.
\newblock Theses, {IME-UFRGS}, 2012.

\bibitem[Vil03]{Villa}
C{\'e}dric Villani.
\newblock {\em Topics in optimal transportation}, volume~58 of {\em Grad. Stud.
  Math.}
\newblock Providence, RI: American Mathematical Society (AMS), 2003.

\bibitem[Wal82]{Walters}
Peter Walters.
\newblock {\em An introduction to ergodic theory}, volume~79 of {\em Grad.
  Texts Math.}
\newblock Springer, Cham, 1982.

\end{thebibliography}
\end{document}